\documentclass[11pt,a4paper,reqno]{amsart}

\pdfoutput=1

\usepackage[utf8]{inputenc}
\usepackage[british]{babel}
\usepackage{amsmath,amssymb,amsthm}
\usepackage{thmtools,thm-restate}
\usepackage{hyperref,float,caption}
\hypersetup{
	colorlinks,
	linkcolor={red!60!black},
	citecolor={green!60!black},
	urlcolor={blue!60!black},
}
\usepackage[nameinlink]{cleveref}
\usepackage{geometry}
\usepackage{enumitem}
\usepackage[presets={vec-cev}]{letterswitharrows}
\usepackage[abbrev, msc-links, nobysame]{amsrefs}
\usepackage{tikz}
\usetikzlibrary{fadings}
\usetikzlibrary{arrows}
\usetikzlibrary {arrows.meta}

\usetikzlibrary{decorations.pathmorphing}

\tikzset{
	vertex/.style={circle, draw, minimum size=1.5em},
	edge/.style={->, > = latex'}
}

\usetikzlibrary{arrows}

\tikzset{
	dot/.style = {circle, draw = black, fill, minimum size=#1,
		inner sep=0pt, outer sep=0pt},
	dot/.default = 3.8pt
}
\usetikzlibrary{angles,quotes}

\usepackage{caption}
\usepackage{subcaption}

\usepackage{xcolor}
\definecolor{CornflowerBlue}{rgb}{0.39, 0.58, 0.93}
\definecolor{LavenderMagenta}{rgb}{0.93, 0.51, 0.93}
\definecolor{PastelOrange}{rgb}{1.0, 0.7, 0.28}

\newcommand{\x}[1]{\vecev{#1}}

\def\N{\mathbb N}

\newcommand{\Down}[1]{\lceil #1 \rceil}
\newcommand{\Up}[1]{\lfloor #1 \rfloor}

\newtheorem{theorem}{Theorem}[section]
\newtheorem{lemma}[theorem]{Lemma}

\newtheorem{proposition}[theorem]{Proposition}
\newtheorem{problem}[theorem]{Problem}

\newtheorem{claim}{Claim}[theorem]
\theoremstyle{definition}

\newtheorem*{theorem*}{Theorem}

\crefname{enumi}{}{}
\crefname{enumii}{}{}
\crefformat{enumi}{#2#1#3}
\crefformat{enumii}{#2#1#3}
\Crefformat{enumi}{#2#1#3}
\Crefformat{enumii}{#2#1#3}

\crefname{definition}{definition}{definitions}
\crefformat{definition}{#2Definition~#1#3}
\Crefformat{definition}{#2Definition~#1#3}

\crefname{section}{section}{sections}
\crefformat{section}{#2Section~#1#3}
\Crefformat{section}{#2Section~#1#3}

\crefname{subsection}{Subsection}{subsections}
\crefformat{subsection}{#2Subsection~#1#3}
\Crefformat{subsection}{#2Subsection~#1#3}

\crefname{lemma}{lemma}{lemmata}
\crefformat{lemma}{#2Lemma~#1#3}
\Crefformat{lemma}{#2Lemma~#1#3}

\crefname{remark}{remark}{remarks}
\crefformat{remark}{#2Remark~#1#3}
\Crefformat{remark}{#2Remark~#1#3}

\crefname{theorem}{theorem}{theorems}
\crefformat{theorem}{#2Theorem~#1#3}
\Crefformat{theorem}{#2Theorem~#1#3}

\crefname{corollary}{corollary}{corollaries}
\crefformat{corollary}{#2Corollary~#1#3}
\Crefformat{corollary}{#2Corollary~#1#3}

\crefname{figure}{figure}{figures}
\crefformat{figure}{#2Figure~#1#3}
\Crefformat{figure}{#2Figure~#1#3}

\crefname{proposition}{proposition}{propositions}
\crefformat{proposition}{#2Proposition~#1#3}
\Crefformat{proposition}{#2Proposition~#1#3}

\crefname{observation}{observation}{observations}
\crefformat{observation}{#2Observation~#1#3}
\Crefformat{observation}{#2Observation~#1#3}

\crefname{claim}{claim}{claims}
\crefformat{claim}{#2Claim~#1#3}
\Crefformat{claim}{#2Claim~#1#3}

\def\lqedsymbol{\ifmmode$\lrcorner$\else{\unskip\nobreak\hfil
		\penalty50\hskip1em\null\nobreak\hfil$\rule{1.2ex}{1.2ex}$
		\parfillskip=0pt\finalhyphendemerits=0\endgraf}\fi}

\newenvironment{claimproof}[1][\proofname]
{%
	\proof[#1]%
}
{%
	\endproof%
}

\title{Menger's theorem for ends of digraphs}

\author{Florian Reich}
\address{Universit{\"a}t Hamburg, Department of Mathematics, Bundesstra{\ss}e~55 (Geomatikum), 20146~Hamburg, Germany}
\email{florian.reich@uni-hamburg.de}

\begin{document}
	
	\begin{abstract}
		Polat generalised Menger's theorem -- the maximum number of vertex-disjoint paths between two sets $A$ and $B$ equals the minimum size of an $A$--$B$~separator -- to ends of undirected graphs.
		
		In this paper we extend Menger's theorem to ends of digraphs.
		As an application, we characterise the combined degree of ends of digraphs.
	\end{abstract}
	
	\maketitle
	
	\section{Introduction}
	Menger's theorem~\cite{menger1927allgemeinen} is a classical result in graph theory which states that the maximum number of vertex-disjoint paths between two sets of vertices $A$ and $B$ is equal to the minimum size of an $A$--$B$~separator in every finite or infinite graph.
	Over the past 50 years, significant progress has been made in extending Menger's theorem to ends of undirected graphs.
	\emph{Ends} are one of the most important concepts in infinite graph theory and can be seen as points at infinity towards which rays converge.
	More formally, ends are defined as equivalence classes of rays, where two rays $R_1$ and $R_2$ are equivalent if there exist infinitely many disjoint $R_1$--$R_2$~paths \cite{diestel2024graph}.
	
	Zelinka \cite{zelinka1970uneigentliche} and Halin~\cite{halin1974note} started investigating Menger's theorem for ends in the context of locally finite graphs.
	Polat~\cites{polat1979aspects,polat1991mengerian,polat1994minimax} extended their results to arbitrary undirected graphs and proved:
	\begin{theorem}[\cite{polat1991mengerian}*{Theorem 3.4}] \label{polat}
		Let $G$ be an undirected graph and let $A$ and $B$ be sets of vertices and ends of $G$ such that $(A,B)$ is dispersed.
		Then the maximum number of vertex-disjoint $\hat{A}$--$\hat{B}$~tracks in $G$ is equal to the minimum size of an $A$--$B$~separator in $G$.
	\end{theorem}
	\noindent
	An $A$--$B$~track $T$ is an $A$--$B$~path, a ray in an end in $B$ that starts in a vertex of $A$, a ray in an end in $A$ that starts in a vertex of $B$ or a double ray that splits into two rays, one in an end in $A$ and one in an end in $B$ such that no internal vertex of $T$ is in $A\cup B$.
	A tuple $(A,B)$ is \emph{dispersed} if for every end $a \in A$ there is a finite set $S_a \subseteq V(G)$ such that there is no $a$--$B$~track in $G - S_a$ and for every end $b \in B$ there is a finite set $S_b \subseteq V(G)$ such that there is no $A$--$b$~track in $G - S_b$.
	A set $S \subseteq V(G)$ is an $A$--$B$~separator if there is no $A$--$B$~track in $G - S$ and the tuples $(A,S)$ and $(S,B)$ are dispersed.
	Moreover, $\hat{X}$ is the union of $X$ and all vertices that dominate an end in $X$.
	
	The most recent progress in this field deals with the following strengthening of Menger's theorem for infinite graphs, which was conjectured by Erd\H{o}s~\cite{nash1967infinite} and proved by Aharoni and Berger \cite{aharoni2009menger}:
	There is a family $\mathcal{P}$ of disjoint $A$--$B$~paths and an $A$--$B$~separator that consists of precisely one vertex of each path in $\mathcal{P}$.
	Bruhn, Diestel and Stein~\cite{bruhn2005menger} extended this strengthening to ends of undirected graph in a setting which is inspired by the topological space~$|G|$ and differs slightly from that in~\cref{polat}.
	
	While previous research has focused on ends of undirected graphs, we take on the task of investigating Menger's theorem for ends of digraphs.
	We aim for a generalisation of \cref{polat} to Zuther's~\cite{zuther1998ends} ends of digraphs.
	Zuther's notion of ends of digraphs has recently received increasing attention~\cites{hamann2024end,hamann2024infinite,reich2024halin,hamann2024boundary,craik2016ends} and extends ends of undirected graphs to in-rays and out-rays:
	An \emph{end} of a digraph $D$ is an equivalence class of in- and out-rays, where two in-/out-rays $R_1$ and $R_2$ are \emph{equivalent} if there are infinitely many directed $R_1$--$R_2$~paths and infinitely many directed $R_2$--$R_1$~paths in $D$.
	
	In the same way, we generalise tracks, separators and dispersedness to digraphs.
	For example, an \emph{$A$--$B$~track} $T$ is a directed $A$--$B$~path, an in-ray in an end in $A$ whose endvertex is in $B$, an out-ray in an end in $B$ whose startvertex is in $A$ or a directed double ray whose initial segment is in an end in $A$ and whose terminal segment is in an end in $B$ such that no internal vertex of $T$ is in $A \cup B$.
	
	Unlike Menger's theorem and its strengthening by Aharoni and Berger, \cref{polat} does not transfer verbatim to digraphs (see~\cref{fig:example}).
	This is due to a structural complication of directed ends:
	Distinct ends $\omega_1, \omega_2$ can contain in-/out-rays $R_1 \in \omega_1$ and $R_2 \in \omega_2$ for which there are infinitely many disjoint $R_1$--$R_2$~paths.
	We denote this by $\omega_1 \leq \omega_2$.
	
	\begin{figure}[ht]
		\centering
		\begin{tikzpicture}
			
			\draw[color=PastelOrange, line width=1,rounded corners=0.3cm] (2.5,-0.3) rectangle (3.5,1.3);
			\draw[PastelOrange] (2.2,0.5) node {$A$};
			
			\foreach \x in {2,...,5}{
				\draw[fill] (1.5*\x,0) circle (1pt);
				\draw[fill] (1.5*\x,1) circle (1pt);
				\draw[edge] ({1.5*\x},0.1) to ({1.5*(\x)},0.9);
			}
			
			\foreach \y in {2,...,4}{
				\draw[edge] ({1.5*\y+0.1},0) to ({1.5*(\y+1)-0.1},0);
				\draw[edge] ({1.5*\y+0.1},1) to ({1.5*(\y+1)-0.1},1);
			}
			
			\draw[path fading=east] (7.6,1) to (9,1);			
			\draw[path fading=east] (7.6,0) to (9,0);
			
			\draw (10,1) node {$R_2 \in \omega_2$};
			\draw (10,0) node {$R_1 \in \omega_1$};
		\end{tikzpicture}
		\caption{There is no $A$--$\{\omega_2\}$~separator of size $1$ and there do not exist two disjoint ${A}$--${\{\omega_2\}}$~tracks.}
		\label{fig:example}
	\end{figure}
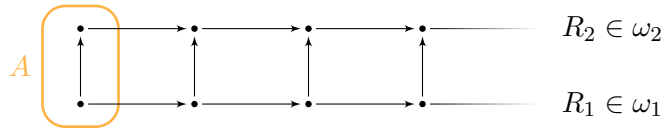
	
	Therefore,
	we have to take into account a more general set of tracks: 
	Let $\Up{A}$ be the union of $A$, the set of ends $\omega$ for which there exists an end $a \in A$ with $a \leq \omega$ and all $A$--\emph{in-dominating}-vertices:
	A vertex $v$ is $A$--\emph{in-dominating} if there exists a ray $R$ in an end of $A$ such that there are infinitely many $R$--$v$~paths that intersect only in $v$.
	Similarly, we define the set $\Down{B}$.
	Since every $A$--$B$~separator must contain a vertex of each $\Up{A}$--$\Down{B}$~track, it is natural to phrase the directed variant of~\cref{polat} as follows:
	
	\begin{restatable}{theorem}{mainTheorem} \label{thm:main}
		Let $D$ be a digraph and let $A$ and $B$ be sets of vertices and ends of $D$ such that $(A,B)$ is dispersed, every end in $A$ contains an in-ray and every end in $B$ contains an out-ray.
		Then the maximum number of disjoint $\Up{A}$--$\Down{B}$~tracks in $D$ is equal to the minimum size of an $A$--$B$~separator in $D$.
	\end{restatable}
	\noindent
	We remark that \cref{thm:main} implies \cref{polat} since $\Up{A} = \hat{A}$ and $\Down{B} = \hat{B}$ in undirected graphs.
	
	The main challenge in proving~\cref{thm:main} arises from the previously discussed complication of ends of digraphs.
	To overcome this, we adapt Galai's~\cite{grunwald1938neuer} alternating path technique so that we can apply it to a special case of~\cref{thm:main}.
	
	As an application of \cref{thm:main} we characterise the combined degree of ends of digraphs, which was recently introduced by Hamann and Heuer~\cite{hamann2024end}, by a family of disjoint tracks that `dominate' the end.
	Furthermore, we will prove that $A$--$B$~double rays are ubiquitous (see~\cite{gut2024ubiquity} for an introduction to ubiquity).
	
	Finally, we raise the question whether \cref{thm:main} can be strengthened in the style of Aharoni and Berger's result:
	
	\begin{problem}
		In the setting of~\cref{thm:main}, does there exist a family $\mathcal{T}$ of disjoint $\Up{A}$--$\Down{B}$~tracks and an $A$--$B$~separator $S$ such that $S$ contains precisely one vertex of each track in $\mathcal{T}$?
	\end{problem}
	
	This paper is structured as follows:
	We begin by introducing basic notations in~\cref{sec:prelim}.
	In~\cref{sec:special_case} we present the adapted alternating path technique and apply it to a special case of~\cref{thm:main}.
	We generalise the results from~\cref{sec:special_case} to arbitrary digraphs in~\cref{sec:generalise}.
	In~\cref{sec:existence} we prove the existence of envelopes and certain $A$--$B$~separators.
	Finally, we prove~\cref{thm:main} in~\cref{sec:main_proof} and use~\cref{thm:main} to characterise the combined end degree in~\cref{sec:combined_end_degree}.

	\section{Basic notation} \label{sec:prelim}
	For standard notation we refer to Diestel's book~\cite{diestel2024graph}.
	Let $\N := \{1, 2, \dots\}$ and $[n] := \{1, \dots, n\}$ for $n \in \N$.
	A digraph $D$ is \emph{outer-locally finite} if every vertex of $D$ has finite out-degree.
	
	A \emph{path} is a digraph whose underlying undirected graph is a path and whose edges are oriented away from its startvertex.
	The terms \emph{startvertex}, \emph{endvertex} and \emph{internal vertex} transfer from its underlying undirected graph to its directed counterpart.
	Given sets of vertices $A$ and $B$, a path $P$ is an \emph{$A$--$B$~path} if $P$ intersects $A$ precisely in its startvertex and $P$ intersects $B$ precisely in its endvertex..
	
	An \emph{in-ray} is a digraph whose underlying undirected graph is a ray and whose edges are oriented towards the unique vertex of degree $1$.
	We refer to the unique vertex of degree~$1$ as the \emph{endvertex} of the in-ray and call all other vertices \emph{internal vertices}.
	Similarly, an \emph{out-ray} is a digraph whose underlying undirected graph is a ray and whose edges are oriented away from the unique vertex of degree~$1$.
	The unique vertex of degree~$1$ is called \emph{startvertex} and we refer to all other vertices as \emph{internal vertices}.
	Given some in-ray (out-ray) $R$, every in-ray (out-ray) in $R$ is a \emph{tail} of $R$.
	
	We say that two in-/out-rays $R_1$ and $R_2$ are equivalent, if there are infinitely many disjoint $R_1$--$R_2$~paths and infinitely many disjoint $R_2$--$R_1$~paths.
	An \emph{end} of a digraph $D$ is an equivalence class of in- and out-rays and we let $\Omega(D)$ be the set of ends of $D$.
	Given two ends $\omega_1$ and $\omega_2$, we write $\omega_1 \leq \omega_2$ if there exist in-/out-rays $R_1 \in \omega_1$ and $R_2 \in \omega_2$ such that there are infinitely many disjoint $R_1$--$R_2$~paths.
	Let $A$ be a set of vertices and ends.
	A vertex $v$ is $A$--\emph{in-dominating} if there exists an in-/out-ray $R$ in an end of $A$ such that there are infinitely many $R$--$v$~paths that intersect only in $v$.
	We set $\Up{A}$ to be the union of $A$, the set of ends $\omega$ for which there exists an end $a \in A$ with $a \leq \omega$ and all $A$--in-dominating vertices.
	Similarly, we define $A$--\emph{out-dominating} vertices and $\Down{A}$.
	
	A \emph{directed double ray} is a digraph whose underlying undirected graph is a double ray and whose edges are consistently oriented.
	We call every vertex of a directed double ray an \emph{internal vertex}.
	Let $A$ and $B$ be sets of vertices and ends.
	An \emph{$A$--$B$~track} $T$ is an $A$--$B$~path, an in-ray in an end in $A$ whose endvertex is in $B$, an out-ray in an end in $B$ whose startvertex is in $A$ or a directed double ray whose initial segment is in an end in $A$ and whose terminal segment is in an end in $B$ such that no internal vertex of $T$ is in $A \cup B$.
	Given some vertex or end $a$, we simply refer to $\{a\}$--$B$~tracks as \emph{$a$--$B$~tracks} and so on.
	
	An \emph{in-arborescence} is a digraph $T$ whose underlying undirected graph is a tree such that all edges of $T$ are oriented towards some fixed vertex in $T$, which we call the \emph{root} of $T$.
	Similarly, we define \emph{out-arborescences}.
	
	A tuple $(A,B)$ of sets of vertices and ends of a digraph $D$ is \emph{dispersed} if for every end $a \in A \cap \Omega(D)$ there is a finite set $S_a$ such that there is no $a$--$B$~track in $D - S_a$ and if for every end $b \in B \cap \Omega(D)$ there is a finite set $S_b$ such that there is no $A$--$b$~track in $D - S_b$.
	A set $S \subseteq V(D)$ is an \emph{$A$--$B$~separator} if there is no $A$--$B$~track in $D - S$ and the tuples $(A,S)$ and $(S,B)$ are dispersed.
	The latter condition of $A$--$B$~separator ensures that $S$ does not `converge' to an end in $A$ or $B$.
	We remark that $(A,S)$ and $(S,B)$ are always dispersed if $S$ is finite.
	
	\section{A special case} \label{sec:special_case}
	In this section, we prove a special case of~\cref{thm:main}:
	\begin{theorem}\label{thm:alternating}
		Let $D$ be an outer-locally finite digraph, let $A \subseteq V(D)$ be finite and let $B \subseteq \Omega(D)$ such that every end in $B$ contains an out-ray.
		Then the maximum number of disjoint $A$--$\Down{B}$~rays is equal to the minimum size of an $A$--$B$~separator.
	\end{theorem}
	\noindent
	For the proof of~\cref{thm:alternating} we adapt Galai's~\cite{grunwald1938neuer} alternating path technique (see \cite{diestel2024graph}*{Section 3.3} for a summary) to out-rays.
	
	Let $D$ be an outer-locally finite digraph and let $\mathcal{R}$ be a finite family of disjoint $A$--$B$~rays for some finite set $A \subseteq V(D)$ and some set $B \subseteq \Omega(D)$.
	We call a sequence $x_0e_1x_1e_2x_2 \dots e_n x_n$ an \emph{$\mathcal{R}$-alternating path} if it induces a trail in the underlying undirected graph that starts in $A \setminus V(\mathcal{R})$ and for every $i < n$:
	\begin{enumerate}[label=(\roman*)]
		\item if $e_i \in E(\mathcal{R})$, then $e_i$ has head in $x_{i-1}$,
		\item if $e_i \notin E(\mathcal{R})$, then $e_i$ has head in $x_{i}$,
		\item if $x_i = x_j$ for $i \neq j$, then $x_i \in V(\mathcal{R})$, and
		\item if $x_i \in V(\mathcal{R})$, then $\{e_{i-1}, e_i\} \cap E(\mathcal{R}) \neq \emptyset$.
	\end{enumerate}
	We call an infinite sequence $S:= v_0e_1v_1e_2v_2 \dots$ an \emph{$\mathcal{R}$--alternating ray} if all initial segments of $S$ are $\mathcal{R}$--alternating paths.
	Furthermore, we say that $S$ \emph{belongs to $\Down{B}$} if either $S$ intersects $\mathcal{R}$ infinitely often or $S$ contains a terminal segment that is an out-ray in $\Down{B}$.
	
	\begin{lemma}\label{lem:alternating_ray}
		If there exists an $\mathcal{R}$--alternating ray that belongs to $\Down{B}$, then there is a family of $|\mathcal{R}|+1$ disjoint $A$--$\Down{B}$~rays.
	\end{lemma}
	
	\begin{proof}
		Let $S$ be some $\mathcal{R}$--alternating ray that belongs to $\Down{B}$.
		We consider the subgraph $H$ induced by the symmetric difference of $E(S)$ and $E(\bigcup \mathcal{R})$.
		Note that $H$ has $|\mathcal{R}|+1$ weak components containing vertices of $A$.
		Note further that each such weak component $C$ is an out-ray $R_C$.
		In particular, $R_C$ has a tail $R_C'$ that intersects $A$ only in its startvertex.
		
		If $S$ intersects $\mathcal{R}$ infinitely often, then every $R_C'$ intersects $\mathcal{R}$ infinitely often.
		In particular, every $R_C'$ is in $\Down{B}$.
		Otherwise, there is a terminal segment $S'$ of $S$ that is a ray in $\Down{B}$.
		Note that every $R_C'$ has a tail in a ray of $\mathcal{R} \cup \{S'\}$ and thus is in $\Down{B}$.
	\end{proof}
	
	\begin{lemma}\label{lem:alternating_ray_or_separator}
		Either there exists a $\mathcal{R}$--alternating ray that belongs to $\Down{B}$ or an $A$--$\Down{B}$~separator of size $|\mathcal{R}|$.
	\end{lemma}
	
	\begin{proof}
		We assume that there does not exist an $A$--$\Down{B}$~separator of size $|\mathcal{R}|$ and show that there is a $\mathcal{R}$--alternating ray that belongs to $\Down{B}$.
		
		\begin{claim}\label{claim1}
			There exists an $\mathcal{R}$-alternating ray belonging to $\Down{B}$ or a sequence of finite, undirected forests $(T_n)_{n \in \N}$ rooted in $A$ together with maps $f_n: V(T_n) \mapsto V(D)$, $g_n: E(T_n) \mapsto E(D)$ such that for every $n > 1$:
			
			\begin{enumerate}[label=(\alph*)]
				\item\label{itm:tree1} the image $f(v_0) g(e_1) f(v_1) \dots g(e_n) g(v_n)$ of every rooted path $v_0 e_1 v_1 \dots e_n v_n$ in $T_n$ is an $\mathcal{R}$-alternating ray,
				\item\label{itm:tree2} $T_{n-1} \subset T_n$, $f_n \vert_{V(T_{n-1})} = f_{n-1}$ and $g_n \vert_{E(T_{n-1}))} = g_{n-1}$,
				\item\label{itm:tree3} $g_n$ is injective, and
				\item\label{itm:tree4} there is a $v$-$\mathcal{R}$~path in $D[f_n(V(T_n)) \setminus f_n(V(T_{n-1}))]$ for every $v \in f_n(V(T_n)) \setminus f_n(V(T_{n-1}))$.
			\end{enumerate}
		\end{claim}
		\begin{claimproof}
			We set $T_1:= (A, \emptyset)$, let $f_1$ be the identity and refer to the vertices in $A$ as the roots.
			We assume that $T_n$ has been constructed for some $n \in \N$.
			For every $R \in \mathcal{R}$ let $u_R$ be the last vertex of $R$ in $f_n(V(T_n))$.
			If $\bigcup_{R \in \mathcal{R}} E(Ru_R) \setminus g_n(E(T_n)) \neq \emptyset$, then there is $uv \in \bigcup_{R \in \mathcal{R}} E(Ru_R) \setminus g_n(E(T_n))$ such that $v \in f_n(V(T_n))$.
			We pick a rooted path $P$ in $T_n$ whose endvertex $z$ satisfies $f_n(z)=v$.
			We set $T_{n+1}$ to be the forest obtained from $T_n$ by adding $uv$ on top of $P$, i.e. adding a vertex $z'$ and an edge $zz'$ with $f_{n+1}(z') = u$ and $f_{n+1}(zz') = uv$.
			Note that this construction satisfies~\labelcref{itm:tree4} since $u,v \in V(\mathcal{R})$.
			
			Otherwise, i.e. if $\bigcup_{R \in \mathcal{R}} E(Ru_R) \subseteq g_n(E(T_n))$, then $f_n(V(T_n)) \cap \bigcup_{R \in \mathcal{R}} V({R})= \bigcup_{R \in \mathcal{R}} V(Ru_R)$ and for every $x \in \bigcup_{R \in \mathcal{R}} V(Ru_r) \setminus \{u_r\}$ there is a rooted path in $T_n$ such that its endvertex $z$ satisfies $f_n(z)=x$ and its final edge $e$ has the property that $g_n(e) \in \bigcup_{R \in \mathcal{R}} E(Ru_R)$.
			Since $\{u_{R}: R \in \mathcal{R}\}$ is not an $A$--$\Down{B}$~separator but $(A,S)$ and $(S,B)$ are dispersed, there exists an $A$--$\Down{B}$~ray $S$ avoiding $\{u_{R}: R \in \mathcal{R}\}$.
			Let $x$ be the last vertex of $S$ in $f_n(V(T_n))$.
			We choose a rooted path $P'$ in $T_n$ whose endvertex $z$ satisfies $f_n(z)=x$ and additionally, if $x \in \bigcup_{R \in \mathcal{R}} V(Ru_R) \setminus \{u_R\}$, its final edge $e$ satisfies $g_n(e) \in \bigcup_{R \in \mathcal{R}} E(Ru_R)$.
			Let $\hat{P}$ be the image of $P'$ as in \labelcref{itm:tree1}.
			
			We can assume that the tail $xS -x$ intersects $\bigcup_{R \in \mathcal{R}} V(R)$, since otherwise the concatenation of $\hat{P}$ and $xS$ is an $\mathcal{R}$-alternating ray belonging to $\Down{B}$ and we are done.
			Let $y$ be the first vertex of $x S -x$ in $\bigcup_{R \in \mathcal{R}} V(R)$.
			Then the concatenation of $\hat{P}$ and $xSy$ is a $\mathcal{R}$-alternating path.
			Let $T_{n+1}$ be the tree obtained from $T_n$ by adding $xSy$ on top of $P'$, i.e. attaching a path $Q$ to the endvertex of $P'$ such that $f_{n+1}$ and $g_{n+1}$ applied to $Q$ induce $xSy$.
			Note that this construction satisfies \labelcref{itm:tree4} since $xSy$ intersects $f_{n+1}(V(T_n))$ only in $x$.
			This completes the construction of $(T_n)_{n \in \N}$.
		\end{claimproof}
		
		By~\Cref{claim1}, we can assume that there is a sequence $(T_n)_{n \in \N}$ as stated.
		Note that every two edge $e,f \in E(T_n)$ with the same tail have the property that $g_n(e)$ and $g_n(f)$ have the same tail.
		This implies that $\bigcup_{n \in \N} T_n$ is outer-locally finite by~\labelcref{itm:tree3} and since $D$ is outer-locally finite.
		By~\labelcref{itm:tree2}, $\bigcup_{n \in \N} T_n$ has an infinite component and thus $\bigcup_{n \in \N} T_n$ contains a rooted ray $S$.
		The rooted ray $S$ induces a $\mathcal{R}$-alternating ray $\hat{S}$ through the maps $(f_n)_{n \in \N}$ and $(g_n)_{n \in \N}$ by~\cref{itm:tree1}.
		
		It remains to prove that $\hat{S}$ belongs to $\Down{B}$.
		If $\hat{S}$ contains infinitely many edges of $\mathcal{R}$, then $\hat{S}$ belongs to $\Down{B}$ by definition.
		Thus there is a tail $S'$ of $S$ such that $(f_n)_{n \in \N}$ and $(g_n)_{n \in \N}$ map $S'$ to a ray $\hat{S}'$ in $D$.
		We show that there is an $\hat{S}'$--$V(\mathcal{R})$~path in $D - X$ for every finite set $X \subseteq V(D)$.
		Then there is $R \in \mathcal{R}$ such that there are infinitely many disjoint $\hat{S}'$--$R$~paths, which implies that $\hat{S}' \in \Down{B}$ and thus $\hat{S}$ belongs to $\Down{B}$.
		
		Let $X \subseteq V(D)$ be an arbitrary finite set.
		Let $m \in \N$ such that $X \cap \bigcup_{n \in \N} f_n(V(T_n)) = X \cap \bigcup_{n \leq m} f_n(V(T_n))$.
		Since $V(\hat{S}')$ is infinite and $f_m(V(T_m))$ is finite, there exists $\ell>m$ for which there is $v \in (f_\ell(V(T_{\ell})) \setminus f_\ell(V(T_{\ell -1 }))) \cap V(\hat{S}')$.
		By \labelcref{itm:tree4}, there is a $v$-$\mathcal{R}$~path $\mathcal{P}$ in $D[f_\ell(V(T_\ell)) \setminus f_\ell(V(T_{\ell-1}))]$.
		Thus $\mathcal{P}$ avoids $X$ by the choice of $m$.
		This completes the proof.
	\end{proof}
	
	Now, \Cref{lem:alternating_ray,lem:alternating_ray_or_separator} prove~\cref{thm:alternating}.
	
	\section{Generalising the special case} \label{sec:generalise}
	In this section we generalise \cref{thm:alternating} to general, not necessarily outer-locally finite, digraphs $D$ and sets $B$ that also contain vertices:
	\begin{lemma}\label{lem:general}
		Let $D$ be a digraph, let $A \subseteq V(D)$ be finite and let $B \subseteq V(D) \cup \Omega(D)$ such that each end in $B$ contains an out-ray.
		Then the maximum number of disjoint $A$--$\Down{B}$~tracks is equal to the minimum size of an $A$--$B$~separator.
	\end{lemma}
	
	\begin{proposition}\label{prop:separator_closure}
		Let $D$ be a digraph and let $A$ and $B$ be sets of vertices and ends of $D$.
		Then every $A$--$B$~separator is also an $\Up{A}$--$\Down{B}$~separator.
	\end{proposition}
	\begin{proof}
		Suppose for a contradiction that there is an $A$--$B$~separator $S$ and an $\Up{A}$--$\Down{B}$~track $T$ such that $S$ and $T$ are disjoint.
		We construct an $A$--$B$~track $T''$ such that $S$ and $T''$ are disjoint, which contradicts that $S$ is an $A$--$B$~separator.
		
		If $T$ starts in $A$, we set $T':= T$.
		Otherwise, let $a \in A$ such that $T$ starts in $\Up{a}$.
		Since $S$ is an $A$--$B$~separator, there is a finite $a$--$S$~separator $S_a$.
		There is an in-ray $R_a \subseteq D - S_a$ in $a$ that contains a vertex of $T$.
		Note that $R_a$ avoids $S$, since $R_a$ avoids $S_a$.
		Then the union $R_a \cup T$ contains an $a$--$\Down{B}$~track $T'$ that avoids $S$.
		
		If $T'$ ends in $B$, the $T'':= T'$ is as desired.
		Otherwise, let $b \in B$ such that $T'$ ends in $\Down{b}$.
		Since $S$ is an $A$--$B$~separator, there is a finite $S$--$b$~separator $S_b$.
		There is a out-ray $R_b \subseteq D - S_b$ in $b$ that contains a vertex of $T'$.
		Note that $R_b$ avoids $S$, since $R_b$ avoids $S_b$.
		Then the union $R_b \cup T'$ contains an $a$--$b$~track $T''$ that avoids $S$.
	\end{proof}
	
	\begin{proof}[Proof of~\cref{lem:general}]
		Let $\mathcal{R}$ be a maximal family of $A$--$\Down{B}$~tracks in $D$.
		By~\cref{prop:separator_closure}, every $A$--$B$~separator has size at least $|\mathcal{R}|$.
		We have to show that there exists an $A$--$B$~separator of size $|\mathcal{R}|$.
		
		First, we reduce this problem to a countable subgraph $D'$ of $D$. 
		Let $\mathcal{X}$ be the set of all $|\mathcal{R}|$--element subsets $X$ of $\bigcup_{R \in \mathcal{R}} V(R)$ for which there exists an $A$--$B$~track $S_X$ in $D-X$.
		We set $D':= \bigcup_{R \in \mathcal{R}} R \cup \bigcup_{X \in \mathcal{X}} S_X$ and let $C$ be the set of endvertices of the paths in $\mathcal{R} \cup \{ S_X: X \in \mathcal{X}\}$ and of ends of the out-rays in $\mathcal{R} \cup  \{ S_X: X \in \mathcal{X}\}$ with respect to $D'$.
		Note that $D'$ is countable.
		Since $D'$ is a subgraph of $D$ and by the choice of $C$, $\mathcal{R}$ is also a maximal family of disjoint $A$--$\Down{C}$~tracks in $D'$.
		We show that there is an $A$--$C$~separator $X$ in $D'$ of size $|\mathcal{R}|$.
		Then $X \subseteq \bigcup_{R \in \mathcal{R}} V(R)$ and thus $X \notin \mathcal{X}$ since otherwise there would exist $S_X$ in $D'$, which is an $A$--$C$~track that avoids $X$.
		This implies that $X$ is an $A$--$B$~separator of size $|\mathcal{R}|$ in $D$, as desired.
		
		Second, we turn $D'$ into an outer-locally finite digraph $D''$ in which all tracks of $\mathcal{R} \cup \{S_X : X \in \mathcal{X}\}$ become out-rays.
		Let $Y$ be the set of endvertices of the paths in $\mathcal{R} \cup \{S_X : X \in \mathcal{X}\}$ and of vertices of infinite out-degree in $D'$.
		Let $D''$ be obtained from the union of $D'$ and a disjoint family of out-rays $(T_y)_{y \in Y}$, where the root of $T_y$ is $y$ and $T_y$ is otherwise disjoint to $D'$, by moving the tails of all edges starting in $y$ to distinct vertices of $T_y$ for every $y \in Y$.
		Let $\hat{C}$ be the set of all ends in $C$ and all ends containing some $T_y$ for $y \in C$.
		Note that $D''$ is outer-locally finite. 
		
		Each $A$--$\Down{C}$~track $Q$ in $D'$ corresponds to some $A$--$\Down{\hat{C}}$~out-ray $\hat{Q}$ in $D''$:
		If $Q$ is an out-ray, let $\hat{Q}$ be obtained from $Q$ by adding some finite path in $T_y$ for every $y \in Y \cap V(Q)$.
		Note that $\hat{Q}$ is contained in some end of $\Down{\hat{C}}$.
		If $Q$ is a path, let $\hat{Q}$ be the path obtained from $Q$ by adding some finite path in $T_y$ for every $y \in Y \cap V(Q)$, as before, and adding $T_{y'}$ at the end, where $y'$ is the endvertex of $Q$.
		Note that $y'$ is either in $C$ or is $C$--out-dominating, which implies that $T_{y'}$, and thus $\hat{Q}$, is in some end of $\Down{\hat{C}}$.
		
		Conversely, each $A$--$\Down{\hat{C}}$~out-ray $\hat{O}$ in $D''$ corresponds to an $A$--$\Down{C}$~track in $D'$:
		Let $O$ be the path obtained from $\hat{O}$ by contracting all finite paths in $\hat{O}$ that are contained in some $T_y$ for $y \in Y$ and possibly contracting some tail of $\hat{O}$ that is contained in some $T_y$ for $y \in Y$.
		If $O$ contains a vertex in $\Down{C}$, then the restriction of $O$ up to the first such $y \in \Down{C}$ is as desired.
		Thus we can assume that all vertices of $O$ are not in $\Down{C}$.
		If $O$ is an out-ray, i.e. no tail of $\hat{O}$ is contained in some $T_y$, then it is in some end of $\Down{C}$ since $\hat{O}$ is in some end of $\Down{\hat{C}}$, as desired.
		Otherwise, i.e. if $\hat{O}$ has a tail in some $T_{y'}$ for $y' \in Y$, the end containing $T_{y'}$ is in $\Down{\hat{C}} \setminus \hat{C}$ since $y' \notin \Down{C}$.
		Thus there are infinitely many $T_{y'}$--$\Tilde{R}$~paths for some out-ray $\Tilde{R}$ in an end of $\hat{C}$.
		Note that the end of $\Tilde{R}$ does not coincide with the end of some $T_{y}$ for $y \in Y$ by construction of $D''$.
		This implies that the end of $\Tilde{R}$ corresponds to an end in $C$.
		Thus $y'$ is $C$--out-dominating in $D'$, which implies that $Oy'$ is as desired.
		
		Thus $\mathcal{R}$ corresponds to a maximal family of disjoint $A$--$\Down{\hat{C}}$~out-rays.
		We apply~\cref{thm:alternating} to $A$ and $\hat{C}$ in $D''$, which provides an $A$--$\hat{C}$~separator $\hat{X}$ in $D''$ of size $|\mathcal{R}|$.
		Then $X:= \{y \in Y: T_y \cap \hat{X} \neq \emptyset\} \cup (\hat{X} \setminus \bigcup_{y \in Y} V(T_y))$ is an $A$--$C$~separator in $D'$ of size $|\mathcal{R}|$, as desired.
		This completes the proof.
	\end{proof}
	
	\section{Separators and envelopes} \label{sec:existence}
	
	In this section we show the existence of the following separator:
	\begin{lemma}\label{lem:existence_separator}
		Let $D$ be a digraph and let $A$ and $B$ be sets of vertices and ends of $D$ such that $(A,B)$ is dispersed.
		Then there exists an $A$--$B$~separator $S$ such that
		\begin{itemize}
			\item for every $a \in A$ there is a finite $a$--$S$~separator contained in $S$, and
			\item for every $b \in B$ there is a finite $S$--$b$~separator contained in $S$.
		\end{itemize}	
	\end{lemma}
	\noindent
	The proof of~\cref{lem:existence_separator} is inspired by the proof of its undirected counterpart~\cite{polat1991mengerian}*{Proposition 0.13 and Theorem 1.2} and builds on the concept of \emph{envelopes}. 
	Envelops have been introduced by Kurkofka and Pitz~\cite{kurkofka2021representation} in the context of undirected graphs, and play an important role in the study of tree-decompositions and the end structure of infinite undirected graphs~\cites{aurichi2024topological,koloschin2023end,pitz2022constructing,albrechtsen2025displaying}.
	We transfer this concept to digraphs.
	
	An \emph{in-envelope} for a set $U \subseteq V(D)$ is a superset $Y \supseteq U$ such that there is a finite $v$--$Y$~separator contained in $Y$ for every $v \in V(D)$ and a finite $\omega$--$Y$~separator contained in $Y$ for every $\omega \in \Omega(D)$ for which there is a finite $\omega$--$U$~separator.
	Similarly, an \emph{out-envelope} for a set $U \subseteq V(D)$ is a superset $Y \supseteq U$ such that there is a finite $Y$--$v$~separator contained in $Y$ for every $v \in V(D)$ and a finite $Y$--$\omega$~separator contained in $Y$ for every $\omega \in \Omega(D)$ for which there is a finite $U$--$\omega$~separator.
	
	Given a set $U \subseteq V(D)$, we say a vertex $v \in V(D) \setminus U$ is \emph{$U$--in-attached}, if there is an infinite family of $v$--$U$~paths that intersect only in $v$.
	Furthermore, we say an out-ray $R$ is \emph{$U$--in-attached}, if there is an infinite family of disjoint $R$--$U$~paths.
	
	\begin{proposition}\label{prop:vertex_or_ray}
		Let $D$ be a digraph, let $U \subseteq V(D)$ and let $v \in V(D) \setminus U$.
		If there does not exist a finite $v$--$U$~separator in $U$, then there is either a $U$--in-attached vertex in $V(D) \setminus U$ or a $U$--in-attached out-ray that avoids $U$.
	\end{proposition}
	\begin{proof}
		We assume that there does not exist a finite $v$--$U$~separator in $U$.
		Then there exists an out-arborescence $T$ rooted in $v$ that intersects $U$ precisely in its leaves, that has infinitely many leaves and such that each inner vertex is contained in some $v$--$U$~path in $T$.
		If $T$ contains a vertex of infinite out-degree, then it is a $U$--in-attached vertex.
		Otherwise, there exists an out-ray in $T$, which is $U$--in-attached and avoids $U$.
	\end{proof}
	
	\begin{lemma}\label{lem:in-envelope}
		Let $D$ be a digraph and let $U \subseteq V(D)$.
		Then there exists an in-envelope for $U$.
	\end{lemma}
	Given an end $\omega \in \Omega(D)$ and some finite set $X \subseteq V(D)$, we set $\vC(\omega, X):= \{ v \in V(D): \text{ there is an } {v} \text{--} {\omega} \text{ track in } D - X \}$ and $\Cv(\omega, X):= \{ v \in V(D): \text{ there is an } {\omega} \text{--} {v} \text{ track in } D - X \}$.
	The proof of~\cref{lem:in-envelope} follows the lines of the proof of its undirected counterpart~\cite{kurkofka2021representation}*{Theorem 3.2}.
	\begin{proof}
		Let $\mathcal{R}$ be maximal family of disjoint $U$--in-attached out-rays.
		Furthermore, let $V$ be the set of $U$--in-attached vertices.
		We show that $Y:=U \cup \bigcup \mathcal{R} \cup V$ is an in-envelope for $U$.
		
		\begin{claim}\label{clm:no_attached_vertex}
			There is no $Y$--in-attached vertex in $V(D) \setminus Y$.
		\end{claim}
		\begin{claimproof}
			Suppose for a contradiction that there is a $Y$--in-attached vertex $v$ in $V(D) \setminus Y$.
			We show that for every finite set $X \subseteq V(D) \setminus \{v\}$ there is a $v$--$U$~path that avoids $X$.
			Then there is an infinite family of $v$--$U$~paths that intersect only in $v$, which implies that $v$ is $U$--in-attached and contradicts the maximality of $V$.
			
			Let $X \subseteq V(D) \setminus \{v\}$ be an arbitrary finite set.
			Then there is an infinite family of $v$--$Y$~paths that intersect only in $v$ and avoid $X$.
			Since $X$ intersects only finitely many elements of $\mathcal{R}$, there is a $v$--$Y$~path that avoids $X$ and ends either in $U \cup V$ or in a vertex of some $R\in \mathcal{R}$ that avoids $X$.
			We can deduce that there is a $v$--$U$~path that avoids $X$.
		\end{claimproof}
		\begin{claim}\label{clm:no_attached_ray}
			There is no $Y$--in-attached out-ray that avoids $Y$.
		\end{claim}
		\begin{claimproof}
			Suppose for a contradiction that there is a $Y$--in-attached out-ray $R$ that avoids $Y$.
			We show that for every finite set $X \subseteq V(D)$ there is an $R$--$U$~path that avoids $X$.
			Then there infinitely many disjoint $R$--$U$~paths, which implies that $R$ is $U$--in-attached and contradicts the maximality of $\mathcal{R}$.
			
			Let $X \subseteq V(D)$ be an arbitrary finite set.
			Then there are infinitely many disjoint $R$--$Y$~paths that avoid $X$.
			Since $X$ intersects only finitely many elements of $\mathcal{R}$, there is an $R$--$Y$~path that avoids $X$ and ends either in $U \cup V$ or in a vertex of some $R\in \mathcal{R}$ that avoids $X$.
			We can deduce that there is an $R$--$U$~path that avoids $X$.
		\end{claimproof}
		
		By~\cref{prop:vertex_or_ray}, \Cref{clm:no_attached_ray,clm:no_attached_vertex} imply that there is a finite $v$--$Y$~separator contained in $Y$ for every $v \in V(D) \setminus Y$.
		Furthermore, $\{v\}$ is a finite $v$--$Y$~separator for every $v \in Y$.
		
		It remains to prove that there is a finite $\omega$--$Y$~separator contained in $Y$ for every $\omega \in \Omega(D)$ for which there is a finite $\omega$--$U$~separator.
		Let $\omega \in \Omega(D)$ be an arbitrary end for which there is a finite $\omega$--$U$~separator $X$.
		Note that $X$ is a $\vC(\omega, X)$--$(V(D) \setminus \vC(\omega, X))$~separator, and thus, in particular, a $\vC(\omega, X)$--$U$~separator.
		Thus $\vC(\omega, X) \cap V = \emptyset$.
		Furthermore, every out-ray $R \in \mathcal{R}$ that intersects $\vC(\omega, X)$ has finite intersection with $\vC(\omega, X)$ and contains an element of $X$.
		We can deduce that $\vC(\omega, X) \cap Y$ is finite.
		For every $x \in X$, let $S_x$ be a finite $x$--$Y$~separator contained in $Y$.
		Note that every $\omega$--$Y$~track ends either in $\vC(\omega, X) \cap Y$ or intersects $X$, in which case it contains an $X$--$Y$~path.
		Then $(\vC(\omega, X) \cap Y) \cup \bigcup_{x \in X} S_x$ is a finite $\omega$--$Y$~separator contained in $Y$.
		This completes the proof.
	\end{proof}
	
	Applying~\cref{lem:in-envelope} to the digraph obtained by reversing all orientations shows:
	
	\begin{lemma} \label{lem:out-envelope}
		Let $D$ be a digraph and let $U \subseteq V(D)$.
		Then there exists an out-envelope for $U$. \qed
	\end{lemma}

	\begin{proof}[Proof of \cref{lem:existence_separator}]
		Let $\mathcal{R}$ be some $\subseteq$--maximal family of disjoint out-rays in $D -(\Down{B} \cap V(D))$ such that each element of $\mathcal{R}$ is contained in some end of $\Down{B}$.
		First, we set $X := (\Down{B} \cap V(D)) \cup \bigcup_{R \in \mathcal{R}} V(R)$ and note that there is no $A$--$B$~track in $D-X$ by maximality of $\mathcal{R}$.
		We show that there is a finite $a$--$X$~separator for every $a \in A\cap \Omega(D)$.
		
		For $a \in A \cap \Omega(D)$ let $U_a$ be a finite $a$--$B$~separator.
		There are at most $|U_a|$ elements of $\mathcal{R}$ that intersect $\Cv(a,U_a)$ and each such element intersects $\Cv(a,U_a)$ finitely.
		Then the union of $U_a$ and all these vertices is a finite $a$--$X$~separator.
		
		Second, we apply~\cref{lem:in-envelope} to $X$, which provides an in-envelope $X'$ for $X$.
		Thus there is a finite $a$--$X'$~separator in $X'$ for every $a \in A$.
		Let $V_a$ be a minimal $a$--$X'$~separator in $X'$, i.e. $V_a=\{x \in X': \text{ there is an } A \text{--} X' \text{ track ending in } x \}$, for every $a \in A$.
		
		Third, we set $V := \bigcup_{a \in A} V_a$.
		By the choice of $V_a$ and since there is no $A$--$B$~track in $D - X$, there is no $A$--$B$~track in $D - V$.
		Let $b \in B \cap \Omega(D)$ be arbitrary and let $Y_b$ be an arbitrary finite $A$--$b$~separator.
		We prove that $V \cap \vC(b, Y_b)$ is finite.
		
		Suppose for a contradiction that $V \cap \vC(b, Y_b)$ is infinite. 
		By the construction of $V$ and since every $A$--$\vC(b, Y_b)$~track contains a vertex of $Y_b$, there is $y \in Y_b \setminus V$ such that there is an $a$--$y$~track in $D - V$ for some $a \in A$ and there are infinitely many $y$--$V$~paths that end in distinct vertices.
		This contradicts the choice of the finite set $V_a$ and thus proves that $V \cap \vC(b, Y_b)$ is finite.
		
		Fourth, we set $Z := \bigcup_{a \in A \cap \Omega(D)} \Cv(a,V_a)$ and let $\bar{V} := V \cup ((A \cap V(D)) \setminus Z)$.
		Note that $((A \cap V(D)) \setminus Z) \cap \vC(b, Y_b) = \emptyset$ for every $b \in B \cap \Omega(D)$ and every finite $A$--$b$~separator $Y_b$.
		Thus $\bar{V}\cap \Cv(b,Y_b)$ is finite, which implies that $Y_b \cup (\bar{V}\cap \Cv(b,Y_b))$ is a finite $\bar{V}$--$b$~separator for $b \in B \cap \omega(D)$.
		Note that $Z \cap B = \emptyset$ by the choice of $Z$.
		
		We apply~\cref{lem:out-envelope} to $\bar{V}$ in the subgraph $D - Z$ to obtain an out-envelope $S$ for $\bar{V}$ in $D -Z$.
		We show that $S$ is the desired $A$--$B$~separator.
		Note that there is no $A$--$B$~track in $D - S$ since $V \subseteq S$.
		Since $S$ is an out-envelope for $\bar{V}$ in $D -Z$, there is a finite $S$--$b$~separator $W_b$ in $D- Z$ that is contained in $S$ for every $b \in B$.
		Note that $W_b$ is also an $S$--$b$~separator in $D$ since every $Z$--$b$~track in $D$ intersects $V\subseteq S$.
		
		It remains to prove that there is a finite $a$--$S$~separator contained in $S$ for every $a \in A$.
		For every $a \in A \cap \Omega(D)$, the set $V_a$ is a finite $a$--$S$~separator since $S \cap \Cv(a,V_a) = \emptyset$.
		Similarly, for every $a \in A \cap V(D) \cap Z$, there is $a' \in A \cap \Omega(D)$ such that $a \in \Cv(a',V_{a'})$, which implies that $V_{a'}$ is a finite $a$--$S$~separator in $S$.
		Finally, for every $a \in (A \cap V(D))\setminus Z \subseteq S$, the set $\{a\}$ is a finite $a$--$S$~separator in $S$.
		This completes the proof.
	\end{proof}
	
	\section{Proof of the main theorem} \label{sec:main_proof}
	
	We begin by proving:
	\begin{theorem}\label{thm:small_separator}
		Let $D$ be a digraph and let $A$ and $B$ be sets of vertices and ends of $D$ such that $(A,B)$ is dispersed.
		For every $\subseteq$-maximal family $\mathcal{F}$ of disjoint ${A}$--${B}$~tracks in $D$, there is an $A$--$B$~separator that is finite, if $\mathcal{F}$ is finite, and is of size $|\mathcal{F}|$ otherwise.
	\end{theorem}
	\noindent
	From~\cref{thm:small_separator} we deduce, first, that the main theorem is well-defined (\cref{prop:well_defined}), second, that the main theorem (\cref{thm:main}) holds true, and, third, that $A$--$B$~double rays are ubiquitous (\cref{prop:ubiquity}).
	The proofs of~\Cref{thm:small_separator,thm:main} follow the lines of Polat's proof of \cref{polat}~\cite{polat1991mengerian}*{Theorem 3.4}.
	
	\begin{proof}
		By~\cref{lem:existence_separator}, there exists an $A$--$B$~separator $S$ that contains a finite $a$--$S$~separator $S_a$ for every $a \in A$ and a finite $S$--$b$~separator $S_b$ for every $b \in B$.
		For $F \in \mathcal{F}$ let $a(F)$ be the element of $A$ in which $F$ starts and let $b(F)$ be the element of $B$ in which $F$ terminates.
		Since $S_{a(F)}$ and $S_{b(F)}$ are finite, there exists a finite path $P(F)$ in $F$ that contains all vertices of $V(F) \cap (S_{a(F)} \cup S_{b(F)})$.
		We set $S':= \bigcup_{F \in \mathcal{F}} S_{a(F)} \cup S_{b(F)} \cup V(P(F))$ and show that $S'$ is an $A$--$B$~separator.
		Note that $S'$ is finite, if $\mathcal{F}$ is finite, and is of size $|\mathcal{F}|$ otherwise.
		
		We begin by proving that there is no $A$--$B$~track in $D - S'$.
		Suppose for a contradiction that there is an $A$--$B$~track $T$ in $D - S'$.
		By maximality of $\mathcal{F}$ there is $F \in \mathcal{F}$ such that there exists some vertex $x \in V(F) \cap V(T)$.
		Since $V(F) \cap (S_{a(F)} \cup S_{b(F)}) \subseteq V(P(F)) \subseteq S'$, there is either an $a(F)$--$x$~track or an $x$--$b(F)$~track in $F - (S_{a(F)} \cup S_{b(F)})$.
		In the former case, the union $T \cup (F - (S_{a(F)} \cup S_{b(F)})) \subseteq D - S_{a(F)}$ contains an $a(F)$--$B$~track, which contradicts that $S_{a(F)}$ is an $a(F)$--$B$~separator.
		In the latter case, the union $T \cup (F - (S_{a(F)} \cup S_{b(F)})) \subseteq D - S_{b(F)}$ contains an $A$--$b(F)$~track, which contradicts that $S_{b(F)} $ is an $A$--$b(F)$~separator.
		
		It remains to show that $(A,S')$ and $(S',B)$ are dispersed.
		Let $a \in A$ be arbitrary.
		Then there is a finite $a$--$S$~separator $S_a$.
		We show that the finite set $S_a':= S_a \cup \{V(P(F)): F \in \mathcal{F} \text{ with } F \cap S_a \neq \emptyset\}$ is an $a$--$S'$~separator.
		Suppose for a contradiction that there is an $a$--$S'$~track $T'$ that avoids $S_a'$.
		If $T'$ ends in $\bigcup_{F \in \mathcal{F}} S_{a(F)} \cup S_{b(F)}$, then $T'$ contradicts that $S_a \subseteq S_a'$ is an $a$--$S$~separator and that $\bigcup_{F \in \mathcal{F}} S_{a(F)} \cup S_{b(F)} \subseteq S$.
		Otherwise, there is $F \in \mathcal{F}$ such that $T'$ ends in $V(P(F))$.
		By the choice of $S_a'$, $V(P(F)) \cap S_a = \emptyset$.
		Thus $T'$ can be extended to some $a$--$S_{b(F)}$~track that avoids $S_a$, contradicting the previous argument.
		Thus $(A,S')$ is dispersed.
		Similarly, we can show that $(S',B)$ is dispersed.
		This completes the proof.
	\end{proof}
	
	\begin{proposition}\label{prop:well_defined}
		Let $D$ be a digraph and let $A$ and $B$ be sets of vertices and ends of $D$ such that $(A,B)$ is dispersed.
		Then there is a family of disjoint $A$--$B$~tracks of size $\sup\{ |\mathcal{F}|: \mathcal{F} \text{ family of disjoint } A \text{--} B \text{ tracks}  \}$.
	\end{proposition}
	\begin{proof}
		Let $\mathcal{T}$ be a $\subseteq$-maximal family of disjoint $A$--$B$~tracks.
		By~\cref{thm:small_separator}, there exists an $A$--$B$~separator $S$ that is finite if $\mathcal{T}$ is finite and has size $|\mathcal{T}|$ if $\mathcal{T}$ is infinite.
		In the former case, $S$ witnesses that $\sup\{ |\mathcal{F}|: \mathcal{F} \text{ family of disjoint } A \text{--} B \text{ tracks}  \}$ is finite.
		Then there exists a family of disjoint $A$--$B$~tracks of size $\sup\{ |\mathcal{F}|: \mathcal{F} \text{ family of disjoint } A \text{--} B \text{ tracks}  \}$.
		In the latter case, $S$ witnesses that $\sup\{ |\mathcal{F}|: \mathcal{F} \text{ family of disjoint } A \text{--} B \text{ tracks}  \}$ is of size $|\mathcal{T}|$ and thus $\mathcal{T}$ is as desired. 
	\end{proof}
	
	\mainTheorem*
	
	We note that $(\Up{A}, \Down{B})$ is dispersed if $(A,B)$ is dispersed.
	\begin{proof}
		By~\cref{prop:separator_closure}, every $A$--$B$~separator is an $\Up{A}$--$\Down{B}$~separator.
		Thus every $A$--$B$~separator has at least the size of the maximum number of disjoint $\Up{A}$--$\Down{B}$~tracks.
		
		Conversely, if the maximum number of disjoint $\Up{A}$--$\Down{B}$~tracks is infinite, \cref{thm:small_separator} ensures the existence of an $\Up{A}$--$\Down{B}$~separator of the same size, which is in particular an $A$--$B$~separator.
		
		If the maximum number of disjoint $\Up{A}$--$\Down{B}$~tracks is finite, \cref{thm:small_separator} ensures the existence of a finite $\Up{A}$--$\Down{B}$~separator.
		Let $S$ be an $\Up{A}$--$\Down{B}$~separator of minimum size.
		Then every $\Up{A}$--$S$~separator and every $S$--$\Down{B}$~separator has size at least $|S|$.
		Thus we can apply~\cref{lem:general} to $S$ and $\Down{B}$ in $D$ to obtain a family $\mathcal{T}$ of disjoint $S$--$\Down{B}$~tracks of size $|S|$.
		Furthermore, \cref{lem:general} applied to $\Up{A}$ and $S$ in the digraph obtained from $D$ by reversing all orientations ensures the existence of a family $\mathcal{T}'$ of disjoint $\Up{A}$--$S$~tracks in $D$ of size $|S|$.
		Note that the elements of $\mathcal{T}$ and $\mathcal{T}'$ intersect only in $S$ since $S$ is an $\Up{A}$--$\Down{B}$~separator.
		By concatenation of the elements of $\mathcal{T}$ and $\mathcal{T}'$ that intersect in $S$, we obtain $|S|$ many disjoint $\Up{A}$--$\Down{B}$~tracks.
		This proves that $S$ has the same size as a maximum family of disjoint $\Up{A}$--$\Down{B}$~tracks.
	\end{proof}
	
	Finally, we prove that $A$--$B$~tracks are \emph{ubiquitous}:
	Recall from~\cite{gut2024ubiquity} that a digraph $H$ is \emph{ubiquitous} if every digraph $D$ containing $n$ disjoint copies of $H$ for every $n \in \N$ contains infinitely many disjoint copies of $H$.
	More precisely, we show:
	\begin{lemma}\label{prop:ubiquity}
		Let $D$ be a digraph and let $A$ and $B$ be sets of vertices and ends of $D$ such that $(A,B)$ is dispersed.
		If there are $n$ disjoint $A$--$B$~tracks in $D$ for every $n \in \N$, then there are infinitely many disjoint $A$--$B$~tracks in $D$.
	\end{lemma}
	
	\begin{proof}
		Let $D$ be a digraph such that there are $n$ disjoint $A$--$B$~tracks for every $n \in \N$.
		Let $\mathcal{F}$ be a $\subseteq$-maximal family of disjoint $A$--$B$~tracks.
		By~\cref{prop:well_defined}, $\mathcal{F}$ is infinite, which completes the proof.
	\end{proof}
	
	\section{An application: the combined end degree} \label{sec:combined_end_degree}
	Given some end $\omega$ of an undirected graph, the \emph{end degree} of $\omega$ is the maximum number of disjoint rays in $\omega$.
	Moreover, the \emph{combined end degree} of $\omega$ is defined as the sum of the end degree of $\omega$ and the number of $\omega$-dominating vertices \cite{gollin2022characterising}.
	The combined end degree can be characterised by the existence of an $\omega$-defining sequence of precisely that order \cite{gollin2022characterising}*{Corollary 5.8}.
	
	Hamann and Heuer~\cite{hamann2024end} recently introduced \emph{end degrees} for digraphs and adapted the characterisation of the combined end degree to the directed setting~\cite{hamann2024end}*{Theorem 5.5}:
	The \emph{degree} $d^-(\omega)$ of an end $\omega$ is the maximum number of disjoint out-rays in $\omega$.
	The \emph{combined end degree} is defined as
	\[\Delta^-(\omega):= d^-(\omega) + \inf\{S \subseteq V(D): S \text{ is } \omega \text{-separating}\},
	\]
	where $\Delta^-(\omega) \in \N \cup \{\infty\}$.
	A set $S \subseteq V(G)$ is \emph{$\omega$-separating}\footnote{Hamann and Heuer called this property `separating $\omega^- \cup \operatorname{dom}(\omega)$ from $\omega$' \cite{hamann2024end}. However, we use a different notation to avoid confusion.} if for every out-ray $R \in \omega$ there exists a tail $Q$ such that $S$ intersects
	\begin{itemize}
		\item every path starting in $V(Q)$ and ending in an $\omega$--out-dominating vertex, and
		\item every out-ray of some end in $\{\alpha: \alpha \leq \omega\} \setminus \{\omega\}$ that starts in $V(Q)$.
	\end{itemize}
	
	We take a different perspective on the combined end degree in digraphs in terms of `dominating' tracks:
	
	\begin{lemma}\label{lem:combined_end_degree}
		Let $\omega$ be an end of a digraph $D$ such that $d^-(\omega) \geq 1$.
		Then $\Delta^-(\omega)$ is equal to the sum of $d^-(\omega)$ and the maximum number of disjoint $\Up{\omega}$--$\mathring{\Down{\omega}}$~tracks.
	\end{lemma}
	\noindent
	Here, $\mathring{\Down{\omega}}:= {\Down{\omega}} \setminus \{\omega\}$.
	\begin{proposition}\label{prop:separating_comparison}
		Let $\omega$ be an end of a digraph $D$ such that $d^-(\omega) \geq 1$ and there is a finite $\omega$-separating set.
		Then a finite set of vertices is $\omega$-separating if and only if it is an $\Up{\omega}$--$\mathring{\Down{\omega}}$~separator.
	\end{proposition}
	\begin{proof}
		Let $S$ be an arbitrary $\omega$-separating set and suppose for a contradiction that there is an $\Up{\omega}$--$\mathring{\Down{\omega}}$~track $T$ that avoids $S$.
		Furthermore, let $R$ be an arbitrary out-ray in $\omega$ and let $Q$ be a tail of $R$ witnessing that $S$ is $\omega$-separating.
		Note that there is a $Q$--$T$~path $P$ that avoids $S$, which implies that $P \cup T$ avoids $S$.
		The union $P \cup T$ contains either a path starting in $V(Q)$ and ending in an $\omega$--out-dominating vertex or an out-ray of some end in $\{\alpha: \alpha \leq \omega\} \setminus \{\omega\}$ that starts in $V(Q)$, a contradiction to the choice of $Q$.
		
		Now, let $\hat{S}$ be an arbitrary $\Up{\omega}$--$\mathring{\Down{\omega}}$~separator and suppose for a contradiction that $\hat{S}$ is not $\omega$-separating.
		Then there exists an out-ray $\hat{R}$ such that for every tail $\hat{Q}$ of $\hat{R}$ there exists either a path $P(\hat{Q})$ starting in $V(\hat{Q})$ and ending in an $\omega$--out-dominating vertex that avoids $\hat{S}$ or an out-ray $P(\hat{Q})$ of some end in $\{\alpha: \alpha \leq \omega\} \setminus \{\omega\}$ that starts in $V(\hat{Q})$ and avoids $\hat{S}$.
		
		Let $U$ be a finite $\omega$-separating set and let $\Tilde{R}$ be an tail of $\hat{R}$ witnessing that $U$ is $\omega$-separating.
		Then the path/out-ray $P(\Tilde{Q})$ intersects $U$ for every tail $\Tilde{Q}$ of $\Tilde{R}$.
		Thus some vertex $u \in U$ is contained in infinitely many $P(\Tilde{Q})$.
		This implies that there exists an in-arborescence $A$ rooted in $u$ with infinitely many vertices in $\Tilde{R}$, that avoids $\hat{S}$ and such that each vertex of $A$ in contained in some path from $\Tilde{R}$ to $u$ in $A$.
		Then $A$ contains either a vertex of infinite in-degree or an in-ray.
		In the former case, this vertex is $\omega$--in-dominating.
		In the latter case, there are infinitely many disjoint paths from $\Tilde{R}$ to this in-ray and thus this in-ray is contained in some end of $\Up{\omega}$.
		In both cases, $A$ contains an $\Up{\omega}$--$U$~track $\hat{T}$ that avoids $\hat{S}$.
		We can deduce that the union of $\hat{T}$ and some $P(\Tilde{Q})$ contains an $\Up{\omega}$--$\mathring{\Down{\omega}}$ track that avoids $\hat{S}$, a contradiction.
	\end{proof}

	\begin{proof}[Proof of~\cref{lem:combined_end_degree} if there is a finite $\omega$-separating set]
		By~\cref{prop:separating_comparison}, the minimum size of an $\omega$-separating set is equal to the minimum size of an $\Up{\omega}$--$\mathring{\Down{\omega}}$~separator.
		In particular, there exists a finite $\Up{\omega}$--$\mathring{\Down{\omega}}$~separator, which implies that ($\Up{\omega}$,$\mathring{\Down{\omega}}$) is dispersed.
		Thus we can apply~\cref{thm:main}, which shows that the minimum size of an $\Up{\omega}$--$\mathring{\Down{\omega}}$~separator is equal to the maximum number of disjoint $\Up{\omega}$--$\mathring{\Down{\omega}}$~tracks.
	\end{proof}
	
	\begin{proposition}\label{prop:finite_separating}
		Let $\omega$ be an end of a digraph $D$.
		If $d^-(\omega)$ is finite, then there exists a finite $\omega$-separating set.
	\end{proposition}
	
	\begin{proof}
		We suppose for a contradiction that $k:=d^-(\omega)$ is finite and that there is no finite $\omega$-separating set.
		We can assume that $k \neq 0$, since the empty set is $\omega$-separating if $k = 0$.
		Let $\mathcal{R}$ be a family of $k$ many disjoint out-rays in $\omega$.
		For every $n \in \N$ we construct a family $(P_i^n: i \in [k+1])$ of disjoint paths in $D$ such that:
		\begin{enumerate}[label=(\alph*)]
			\item\label{itm:construct_1} $P_i^n$ is a proper initial segment of $P_i^{n+1}$ for every $i \in [k+1]$,
			\item\label{itm:construct_2} there exists a tail $R'$ of $R$ that starts in the endvertex of some $P_i^{n}$ and is otherwise disjoint to $\bigcup_{i \in [k+1]} P_i^n$ for every $R \in \mathcal{R}$,
			\item\label{itm:construct_3} the unique endvertex $x^n$ of the paths $(P_i^n: i \in [k+1])$ that is not startvertex of some $\{R': R \in \mathcal{R}\}$ is either $\omega$--out-dominating or the startvertex of some out-ray in some end in $\mathring{\Down{\omega}}$ that intersects $\bigcup_{i \in [k+1]} P_i^n \cup \bigcup \mathcal{R}$ only in $x^n$.
		\end{enumerate}
		Furthermore, we construct families $(S^n: n \in \N)$ and $(T^n: n \in \N)$ of disjoint paths, where $S^n$ is an $(\bigcup \mathcal{R})$--$x^n$~path and $T^n$ is an $x^n$--$(\bigcup \mathcal{R})$~path.
		
		Since the empty set is not $\omega$-separating, there exists a vertex $x^1$ that is either $\omega$--out-dominating or the startvertex of some out-ray in some end in $\mathring{\Down{\omega}}$ that avoids $\bigcup \mathcal{R}$.
		Let $a_1, \dots, a_k$ be vertices such that for each out-ray $R \in \mathcal{R}$ there is $i \in [k]$ such that $a_i \in V(R)$ and $x^1 \notin V(a_iR)$.
		We set $P_i^1:= \{a_i\}$ for $i \in [k]$ and $P_{k+1}^1 := \{x^1\}$.
		
		We assume that $(P_i^n: i \in [k+1])$ have been constructed for some $n \in \N$.
		Let $a_R \in V(R)$ such that $a_R R$ is the unique tail of $R$ as defined in \labelcref{itm:construct_2} for every $R \in \mathcal{R}$, i.e. $a_R$ is the endvertex of some $P_i^n$.
		Since the finite set $X:= V(\bigcup_{m \in [n]} (S^m \cup T^m) \cup \bigcup_{i \in [k+1]} P_i^{n})$ is not $\omega$-separating and by the definition of ends, there is either a path $Q$ ending in an $\omega$--out-dominating vertex or an out-ray $Q$ in some end in $\mathring{\Down{\omega}}$ such that $Q$ intersects $(\bigcup_{R \in \mathcal{R}} a_R R)$ precisely in its startvertex and avoids $X$.
		
		Let $q$ be the startvertex of $Q$ and let $R_q \in \mathcal{R}$ with $q \in V(R_q)$.
		By the choice of $x^{n}$, there are $k$ many $x^{n}$--$qR_q$~paths $O_1, \dots, O_k$ that avoid $(\bigcup_{i \in [k+1]} P_i^n \setminus \{x_n\}) \cup V(a_{R_q} R_q q)$ and are disjoint in $\bigcup_{R \in \mathcal{R}} V(a_R R)$.
		
		We set $A:=\{a_R: R \in \mathcal{R}\} \cup \{x^n\}$, i.e. $A$ is the set of all endvertices of the paths $(P_i^{n}: i \in [k+1])$.
		Let $b_R \in V(R)$ such that $b_R R$ avoids $\bigcup_{i \in [k]} O_i$ for every $R \in \mathcal{R}$.
		Furthermore, let $x^{n+1}$ be the endvertex of $Q$, if $Q$ is a path, or, otherwise, some vertex in $V(Q)$ such that $x^{n+1} Q$ avoids $\bigcup_{i \in [k]} O_i$.
		We set $B:= \{b_R: R \in \mathcal{R}\} \cup \{x^{n+1}\}$.
		We show that there is no $A$--$B$~separator of size $k$ in $\bigcup_{R \in \mathcal{R}} a_R R b_R \cup Q x^{n+1} \cup \bigcup_{i \in [k]} O_i$.
		
		Suppose for a contradiction, that $Y$ is an $A$--$B$~separator of size $k$.
		Then each of the disjoint paths $(a_R R b_R)_{R \in \mathcal{R}}$ contains precisely one vertex of $Y$.
		Moreover, the path $a_{R_q} R_q q Q x^{n+1}$ shows that the unique vertex in $V(R_q) \cap Y$ is contained in $V(a_{R_q} R_q q)$.
		Then there is $i \in [k]$ such that $O_i$ avoids $Y$ since $\bigcup_{i \in [k]} O_i$ avoids $V(a_{R_q} R_q q)$.
		This implies that $O_i \cup (q R_q - q)$ contains an $A$--$B$~path that avoids $Y$, a contradiction.
		
		Thus, by Menger's theorem, there is a family $\mathcal{F}$ of $k+1$ many disjoint $A$--$B$~paths in $\bigcup_{R \in \mathcal{R}} a_R R b_R \cup Q x^{n+1} \cup \bigcup_{i \in [k]} O_i$.
		Let $P_i^{n+1}$ be the concatenation of $P_i^{n}$ and the unique path in $\mathcal{F}$ that starts in the endvertex of $P_i^{n}$ for every $i \in [k+1]$.
		By the choice of $Q$, there exists an $(\bigcup \mathcal{R})$--$x^{n+1}$~path $S_{n+1}$ and an $x^{n+1}$--$(\bigcup \mathcal{R})$~path $T_{n+1}$ such that $S_{n+1}$ and $T_{n+1}$ avoid $X$.
		This completes the construction.
		
		Then $(S_n)_{n \in \N}$ and $(T_n)_{n \in \N}$ witness that $\bigcup_{n \in \N} P_i^n$ is an out-ray in $\omega$ for every $i \in [k+1]$.
		This implies that $(\bigcup_{n \in \N} P_i^n)_{i \in [k+1]}$ is a family of disjoint rays in $\omega$, which contradicts the choice of $k$ and completes the proof.
	\end{proof}
	
	\begin{proof}[Proof of~\cref{lem:combined_end_degree} if there is no finite $\omega$-separating set]
		By~\cref{prop:finite_separating}, $d^-(\omega)$ is infinite.
		This implies that both $\Delta^-(\omega)$ and the sum of $d^-(\omega)$ and the maximum number of disjoint $\Up{\omega}$--$\mathring{\Down{\omega}}$~tracks are infinite, and thus equal.
	\end{proof}
	
	We remark that if we allow $R$ to be either an in-ray or an out-ray in the definition of $\omega$-separating, then \cref{lem:combined_end_degree} also holds true for ends $\omega$ with $d^-(\omega) = 0$.
	
	\section*{Acknowledgement}
	The author gratefully acknowledges support by a doctoral scholarship of the Studienstiftung des deutschen Volkes.
	
	\bibliography{ref.bib}
	
\end{document}